\documentclass[11pt]{amsart}
\def\draft{y}
\usepackage{amssymb}
\usepackage{graphicx} 
\usepackage{subfig}
\usepackage{dbnsymb}
\usepackage{enumerate}
\usepackage{mathtools}
\usepackage{color,soul}
\usepackage{tikz}
\usetikzlibrary{matrix}
\usepackage{caption}
\usepackage{amsmath, amsthm, stackrel}
\usepackage{mathtools}
 \usepackage{relsize}
 \usetikzlibrary{cd}
 \usetikzlibrary{decorations.pathreplacing}
 \usepackage{tikz,calc}
\usepackage{color}
\usepackage[margin=1.2in]{geometry}

\usepackage{multicol}

\usepackage{wrapfig}
\usepackage{cutwin}

\usepackage{lmodern}
\usepackage{enumitem}
\usepackage{lipsum}
\usepackage{stackengine}
\usepackage{appendix}
\usepackage{scrextend}
\usepackage[pagebackref,breaklinks=true]{hyperref}\hypersetup{colorlinks,
  linkcolor={green!50!black},
  citecolor={green!50!black},
  urlcolor=blue
}

\newtheorem{thm}{Theorem}[section]

\newtheorem{prop}[thm]{Proposition}
\newtheorem{lem}[thm]{Lemma}

\newtheorem{defn}[thm]{Definition}

\newcommand{\theoremname}{Theorem:}

  \theoremstyle{definition}

  \newtheorem*{claim*}{Claim}
  
  \newtheorem{ex}[thm]{Example}
  
  \newtheorem*{question*}{Question}
  \newtheorem*{answer*}{Answer}
  \newtheorem*{application*}{Application}

  \theoremstyle{remark}
  
  \newtheorem*{rmk*}{Remark}

\usepackage{bbm}
\newcommand{\Q}{\mathbb{Q}}
\newcommand{\R}{\mathbb{R}}

\def\bbR{{\mathbb R}}

\def\draftcut{\if\draft y \cleardoublepage \fi}

\begin{document}

\title{Yarn Ball Knots and Faster Computations}

\author{Dror Bar-Natan}
\address{University of Toronto}
\email{drorbn@math.toronto.edu}
\urladdr{\url{http://www.math.toronto.edu/drorbn}}

\author{Itai Bar-Natan}
\address{University of California, Los Angeles}
\email{itaibn@math.ucla.edu}
\urladdr{}

\author{Iva Halacheva}
\address{Northeastern University}
\email{i.halacheva@northeastern.edu}
\urladdr{\url{https://sites.google.com/site/ivahalacheva3/}}

\author{Nancy Scherich}
\address{Elon Univeristy}
\email{nscherich@elon.edu}
\urladdr{\url{http://www.nancyscherich.com}}

\keywords{finite type invariants}

\thanks{}

\begin{abstract}
We make use of the 3D nature of knots and links to find savings in computational complexity when computing knot invariants such as the linking number and, in general, most finite type invariants. These savings are achieved in comparison with the 2D approach to knots using knot diagrams.
\end{abstract}

\maketitle
\tableofcontents

\section{Introduction}

\subsection{Motivation}
A recurring question in knot theory is ``Do we have a 3D understanding of a knot invariant?". We often think about knots as planar diagrams and compute invariants from the diagrams. 
In this paper, we consider computation from a planar diagram to be  a ``2D" understanding of the invariant. 
For example, Kauffman~\cite{Kauffman} gave a description of the Jones polynomial using planar diagrams, bringing the understanding of this invariant to 2D.\footnote{In some sense, the original description of the Jones polynomial~\cite{Jones}, by Vaughan Jones and using braids, is 1D. For, while the full 3D rotation group $SO(3)$ acts on 3D space, and the 2D rotation group $SO(2)$ acts on planar diagrams, no continuous symmetries are left when one comes down to braids.}
Witten~\cite{Witten} developed Chern-Simons theory, a $3$-dimensional topological quantum field theory, partially in order to understand the Jones polynomial in a 3D way. However, this theory does not seem to lead to easier computations of the invariant. 
So, while the Chern-Simons theory description is 3D, from the point of view of this paper, it is (at least computationally) still an incomplete 3D understanding of the Jones polynomial.


\begin{figure}[htp]
\centering
\begin{picture}(220,150)
\put(-30,0){\includegraphics[width=.3\textwidth]{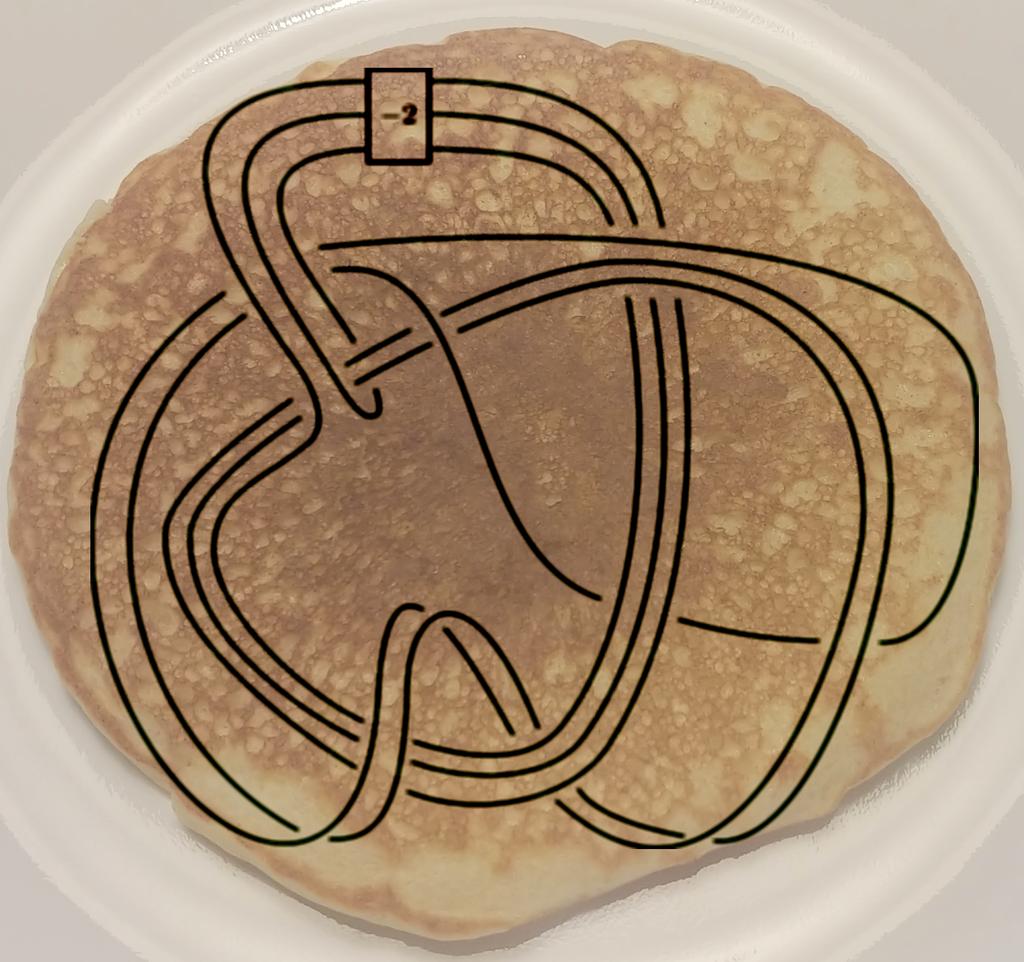}\hspace{.6cm}}
\put(120,0){\includegraphics[width=.3\textwidth]{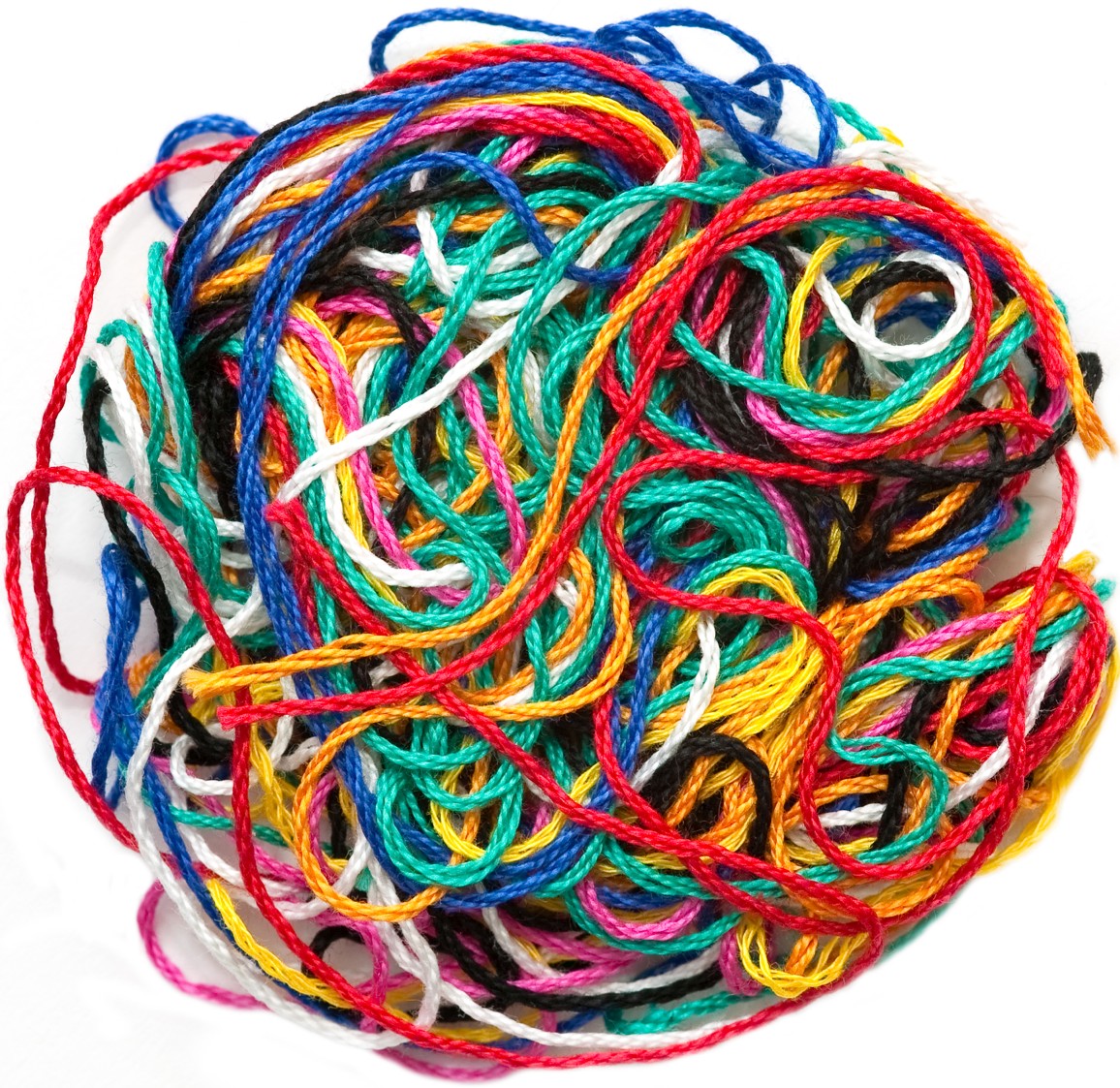}\hspace{.7cm}}
\put(-30,5){(A)}
\put(120,5){(B)}
\end{picture}
\caption{(A)~A planar diagram of a knot in a pancake (knot diagram by Piccirillo~\cite{Piccirillo}). (B) A yarn ball.}  \label{fig:PancakeAndYarn}
\end{figure}

  The complexity of planar knot diagrams is often measured by their number of crossings, denoted by $n$ in this paper, and the computational complexity of a knot invariant is typically described in terms of $n$. Three-dimensionally, planar diagrams can be seen as projections of knots embedded in ``pancakes"-- very flat and  wide subsets of $\bbR^3$, as in Figure \ref{fig:PancakeAndYarn} (A).
This pancake description of a knot is somewhat artificial and does not realistically describe knots that occur in nature. 
For example, knotted DNA is shaped much more like the ball of yarn in Figure  \ref{fig:PancakeAndYarn} (B), than the pancake knot in (A).

The point of view we present in this paper is that knots are three-dimensional and the best way to understand a knot should be three-dimensional. In the subsequent sections, we propose  some new language to aid the knot theory community in discussion surrounding current understanding of an invariant.

\subsection{Yarn balls and pancakes}
For complexity measurements in this paper, we measure only polynomial degree  (i.e.  we ignore constants and $\log (n)$ terms). We write $f(n)\sim g(n)$ to mean there exist natural numbers $c,k,N$ so that for all $n>N$, we have that 
\[\frac{1}{c}g(n)(\log(n))^{-k}<f(n)\stackrel{(1)}{<}cg(n)(\log(n))^k.\]

\noindent For example, for us, $n^4\sim {5.4}\; {n^4}(\log(n))^8$. If the inequality (1) holds true, we write  $g \gtrsim f$.
Also, we write $g\gg f$ to mean that for all constants $c$ and $k$, and for all large enough $n$ we have that $cf(n)(\log
(n))^k<g(n)$.


Given a knot, i.e. a particular embedding of $S^1$ into $3$-space, a knot invariant by definition will output the same value for different diagrams of the knot (i.e. different planar projections). However, the complexity of computing the invariant using a particular algorithm may differ depending on the number of crossings $n$ in the input diagram. 
So, we will measure the computational complexity of an invariant using the number of crossings of the input diagram, rather than the minimal crossing number of the knot type (the knot considered up to ambient isotopy). Similarly, if we start with a 3D presentation of a knot, in a yarn ball or a grid as discussed in Section \ref{sec:GridsAndLinkingNumber}, the complexity of an invariant will be measured by the volume $V$ of the input knot, described below, and not by the minimal volume of the knot type.

The comparison between 3D computational complexity and 2D complexity of a knot invariant relies on the comparison of $V$ to $V^{4/3}$.
The quantity $\sim V^{4/3}$ is the number of crossings in a generic projection of a yarn ball knot. 
We describe this result here and formalize the concepts of a \emph{pancake knot} and \emph{yarn ball knot}.

\begin{defn}
A \emph{pancake} is a thickened disk in $\R^3$ with radius greater than 2 and constant height 2. 
A \emph{pancake knot} is a knot contained in a pancake.
\end{defn}
 
Planar diagrams of knots are projections of pancake knots  in $\R^3$.

\begin{defn}
    A \emph{yarn ball knot} in $\R^3$ is a knotted tube of uniform width 1 in a 3-ball of diameter $L$, such that the volume of the knot (i.e. the volume of the tube) is $\sim$ $L^3$. $L$ is called the diameter of the yarn ball knot. 
\end{defn}

An equivalent notion of a yarn ball knot is a knot embedded in a grid, which we discuss further in Section \ref{sec:GridsAndLinkingNumber}. Theorems \ref{thm:lk} and \ref{Thm:FiniteTypeD} are proved using knots embedded in a grid.

A yarn ball knot in a $3$-ball of volume $V$ has by definition volume and length both $\sim$ $V$.
Most knots appearing in nature are yarn ball knots, as for example the knotted ball of yarn in Figure \ref{fig:diskwithunitsquare}.

 \begin{figure}
    \begin{picture}(280,150)
\put(0,20){\includegraphics[width=.6\textwidth]{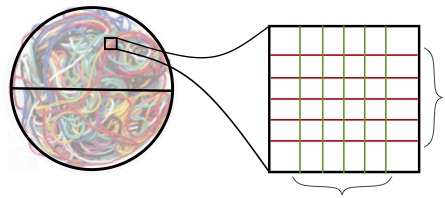}\hspace{.7cm}}
\put(188,7){$\sim \frac{L}{2}$ strands}
\put(265,76){$\sim \frac{L}{2}$ strands }
\put(50,70){\large $L$}
\put(53,107){\textbf{1}}
\put(63,98){\textbf{1}}
\end{picture}
     \caption{Projection of a yarn ball knot to a disk of diameter $L$. To count the crossings in the projection, subdivide the disk into $1\times 1$ squares. Each square will have $\sim \frac{L}{2}$ strands crossing each other, for a total of $\sim {L \choose 2} \sim L^2$ crossings in that square.}
     \label{fig:diskwithunitsquare}
 \end{figure}
 
 By projecting the yarn ball knot in a generic direction onto a disk, and shrinking the width of the tube to $\epsilon$, we attain a planar knot diagram.  The number of crossings of this projection can be estimated by  subdividing the disk into $1\times 1$ squares, as in Figure \ref{fig:diskwithunitsquare}. 
For a sufficiently complicated knot, heuristically we expect most such $1\times 1$ squares  will have $L$ layers of strands above them, and one can expect these strands to cross around $\sim {L \choose 2} \sim L^2$ times in the projected square. 
Since there are $\sim L^2$ squares, the total crossing number of this projection is around $n \sim L^2L^2=L^4=V^{{4}/{3}}$.
We assert that at worst, a knot of volume $V$ can always be projected with $\sim V^{4/3}$ crossings.

Since we aim to have a 3D understanding of knot invariants, we study knots in the shape of the yarn ball as opposed to the pancake. We see that to describe a yarn ball knot of volume (or length) $V$ as a planar diagram, we would need $\sim V^{{4}/{3}}$ crossings. Since $V\ll V^{4/3}$, for large values of $V$ it requires many more bits to describe a yarn ball knot via its projection rather than directly as a yarn ball.

\subsection{Computational complexities of knot invariants}

We use the notation $C_{\zeta}(3D,V)$ and $C_{\zeta}(2D,n)$ to denote the asymptotic time complexities of computing an invariant $\zeta$ via different dimensional methods (3D or 2D) and sizes of input. 
\begin{defn}
     Let $\zeta$ be a knot invariant. The worst-case complexity of computing $\zeta$ on a knot given by a planar diagram with $n$ crossings is denoted by \[C_\zeta(2D,n).\] 
     The worst-case complexity of computing $\zeta$ on a knot given as a yarn ball of volume $V$ is denoted by \[C_\zeta(3D,V).\] 

\end{defn}

However, these abstract complexities, while well defined, are unknown and often unknowable. In practice, one works with a known upper bound for the complexity coming from a current best known algorithm.
To reflect this more practical perspective, we define $FKT_{\zeta}(kD,n)$, or \textbf{f}astest \textbf{k}nown  \textbf{t}ime.

\begin{defn}
    Let $\zeta$ be a knot invariant. $FKT_{\zeta}(kD,n)$ is the computational complexity of the current fastest algorithm known to compute $\zeta$ with $k$-dimensional input of size $n$.
\end{defn}

$FKT_{\zeta}(kD,n)$ is dependent on our current understanding of the computational problem, as the current known fastest algorithm may be improved or supplanted by a better one. In analogy, a current world record may be broken by a faster athlete. To summarise, $C_{\zeta}(kD,n)$ has a fixed value, which can however be very difficult to determine, while the value of $FKT_{\zeta}(kD,n)$ depends on the current best known algorithm and so can be computed using that knowledge. By these definitions, we always have that $C_{\zeta}(kD,n) \leq FKT_{\zeta}(kD,n)$.

\subsection{2D versus 3D algorithms}

Given a yarn ball $K$ of volume $V$, we can always compute an invariant $\zeta$ on $K$ by first projecting to the plane\footnote{The projection itself can be computed quickly, in time $\sim V^{4/3}$, and for all interesting $\zeta$, this extra work is negligible. (The computation can be performed by going over all crossing fields in a grid presentation of the knot, as described in Section \ref{sec:GridsAndLinkingNumber}.)},  obtaining a planar diagram with $\sim V^{4/3}$ crossings, and then computing $\zeta$ using our best 2D techniques. Hence always, $C_{\zeta}(3D,V)\leq C_{\zeta}(2D,V^{4/3})$.
 It is interesting to know when 3D techniques can do even better. Our first main result is to show that for the link invariant the \emph{linking number}, there is a 3D computation technique whose worst-case complexity is faster than every worst-case 2D technique.

\vskip 2mm
\noindent \textbf{Theorem \ref{thm:lk}.} (Proof in Section \ref{sec:GridsAndLinkingNumber})\textit{
Let $lk$ denote the linking number of a 2-component link. Then $C_{lk}(2D,n)\sim n$ while $C_{lk}(3D,V)\sim V$.}\\

The linking number of a link is an example of a \emph{finite type invariant}. 
Finite type invariants underlie many of the classical knot invariants, for instance they give the coefficients of the Jones, Alexander, and more generally HOMPFLY-PT polynomials \cite{BL93, BN1}. We prove the following computational bounds for all finite type invariants.\\

\noindent\textbf{Theorem \ref{Thm:FiniteType2}.} (Proof in Section \ref{sec:FiniteType})\textit{
If $\zeta$ is a finite type invariant of type $d$ then $C_\zeta(2D, n)$ is at most $\sim n^d$.
}\\

\noindent\textbf{Theorem \ref{Thm:FiniteTypeD}.} (Proof in Section \ref{sec:FiniteType})\textit{
If $\zeta$ is a finite type invariant of type $d$ then $C_\zeta(3D, V)$ is at most $\sim V^d$.
}\\

 The actual complexities $C_\zeta(2D,n)$ for some specific though `special' finite type $\zeta$'s, such as the coefficients of the Alexander polynomial, are known to be much smaller. However, for generic $\zeta$'s, these theorems suggest that 3D techniques will be more computationally efficient than 2D ones.  We suspect the upper bounds in Theorem 3.2 and 3.3 can be improved by closer consideration of the counting arguments used in the proofs. However, we view these theorems as a significant starting point that we hope to improve upon in the future and encourage our readers to do the same.

To capture the comparison between the complexity of 2D and 3D techniques, we introduce the following terminology.

 \vskip 3mm
\noindent \textbf{Conversation Starter 1.} \textit{A knot invariant $\zeta$ is said to be \emph{computationally 3D}, or C3D, if}
$$FKT_{\zeta}(3D,V)\ll FKT_{\zeta}(2D,V^{4/3}).$$

\noindent \textit{In other words, $\zeta$ is \textit{C3D} if substantial savings can be made to the computation of $\zeta$ on a yarn ball knot, relative to the complexity of computing $\zeta$ by first projecting the yarn ball to the plane.}\\

As discussed above, the notion of an invariant being computationally 3D that we use here is dependent on the current knowledge of the invariant. As our understanding grows and our computational techniques get better, an invariant might become \textit{C3D}, or lose its \textit{C3D} status. 
However, the question of whether an invariant is \textit{C3D}, as we understand it at a given time, still has merit as it measures our understanding, as a community, of knot theory as a 3D subject.
With this new terminology, Theorem \ref{thm:lk} could be restated as

\vskip 2mm
\noindent \textbf{Theorem \ref{thm:lk}.} (restated)\textit{
The linking number of a 2-component link is C3D.}\\

The results of Theorems~\ref{Thm:FiniteType2} and~\ref{Thm:FiniteTypeD} naively suggest that finite type invariants are also C3D, but the theorems only give one sided bounds.
 Yet, we believe that our naive conclusion remains valid, at least in the form ``most finite type invariants are C3D''.

The opinion in this paper is that in general knot invariants should be \textit{C3D}. Unfortunately, as of the time this paper is written, very few knot invariants are known to be \textit{C3D}. Are the Alexander, Jones, or  HOMFLY-PT polynomials \textit{C3D}? Why or why not? Are the Reshetikhin-Turaev invariants  \textit{C3D}? Are knot homologies \textit{C3D}? While we seem to have a weak understanding of these fundamental invariants from a 3D perspective, this is cause for optimism; there is still much work to be done.

\subsection{Further directions}
The above ideas admit a further generalization or extension. Instead of using computational complexity to compare 2D and 3D understandings of invariants, we can also use the notion of the maximal value of other quantities relating to the size of the knot, which motivates the next conversation starter.\\

\noindent \textbf{Conversation Starter 2.} \textit{If $\eta$ is a stingy quantity (i.e. we expect it to be small for small knots), we say $\eta$ has savings in 3D, or has `S3D' if }
$$M_\eta(3D,V)\ll M_\eta(2D,V^{4/3}),$$

\noindent \textit{where $M_\eta(kD,s)$ is the maximum value of $\eta$ on all knots described $k$-dimensionally and of size $s$.}\\

For example, the hyperbolic volume is a stingy quantity--the more complicated a knot is, the more complicated its complement in $S^3$ will be, which makes the question of putting a hyperbolic structure on it harder. We expect hyperbolic volume to have savings in 3 dimensions.\\

\noindent\textbf{Conjecture} (D. Bar-Natan, van der Veen) Hyperbolic volume has \emph{S3D}.\\

The genus of a knot is another example of a stingy quantity, but we do not know if the genus of a knot has \emph{S3D}, or not. If a knot is given in 3-dimensions, is the best way to find the genus truly to compute the Seifert surface from a projection to 2D, at a great cost? The genus is by all means a 3D property of a knot, and it seems as though it \emph{should} be best computed in a 3D manner. 
 
 We hope that these conversation starters will encourage our readers to think about more 3D computational methods. The remaining two sections of this paper are dedicated to proving Theorems \ref{thm:lk} and \ref{Thm:FiniteTypeD}.\\

\noindent \textbf{Acknowledgements.} This work was partially supported by NSERC grant RGPIN-2018-04350, and by the Chu Family Foundation (NYC). We would like to thank our anonymous referee for helpful comments.

\section{Grid Knots and Linking Number}\label{sec:GridsAndLinkingNumber}

 To emphasize the 3D nature of knots, we think of them as yarn ball knots, instead of as pancake knots. An equivalent notion is ``grid knots'' (see also~\cite{BL12}).  A \emph{grid knot (or link) of size $L$} is a labeled parameterized knot or link embedded as a subset of a grid with side length $L$.  
 The arcs of a grid knot are enumerated in order of the parametrization of the knot along the unit grid segments of the grid lines. 
 For a grid link, the link components are enumerated, and each arc of the link is labeled by a pair $(c,p)$, where $c$ is the arc's component enumeration, and $p$ is the arc's  parametrization enumeration. 
 Some examples of grid knots are shown in Figure \ref{fig:grid_Knot}. Figure \ref{fig:grids} (A) and (B) shows an example of a $4\times 4\times 4$ grid from different perspectives. 

\begin{figure}[htp]
\includegraphics[width=.4\textwidth]{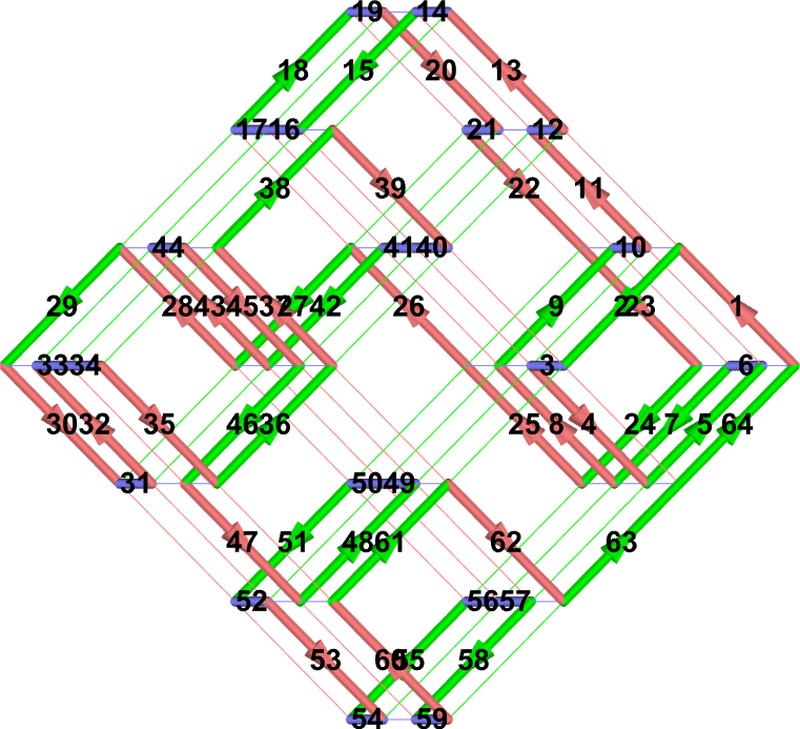}
\includegraphics[width=.4\textwidth]{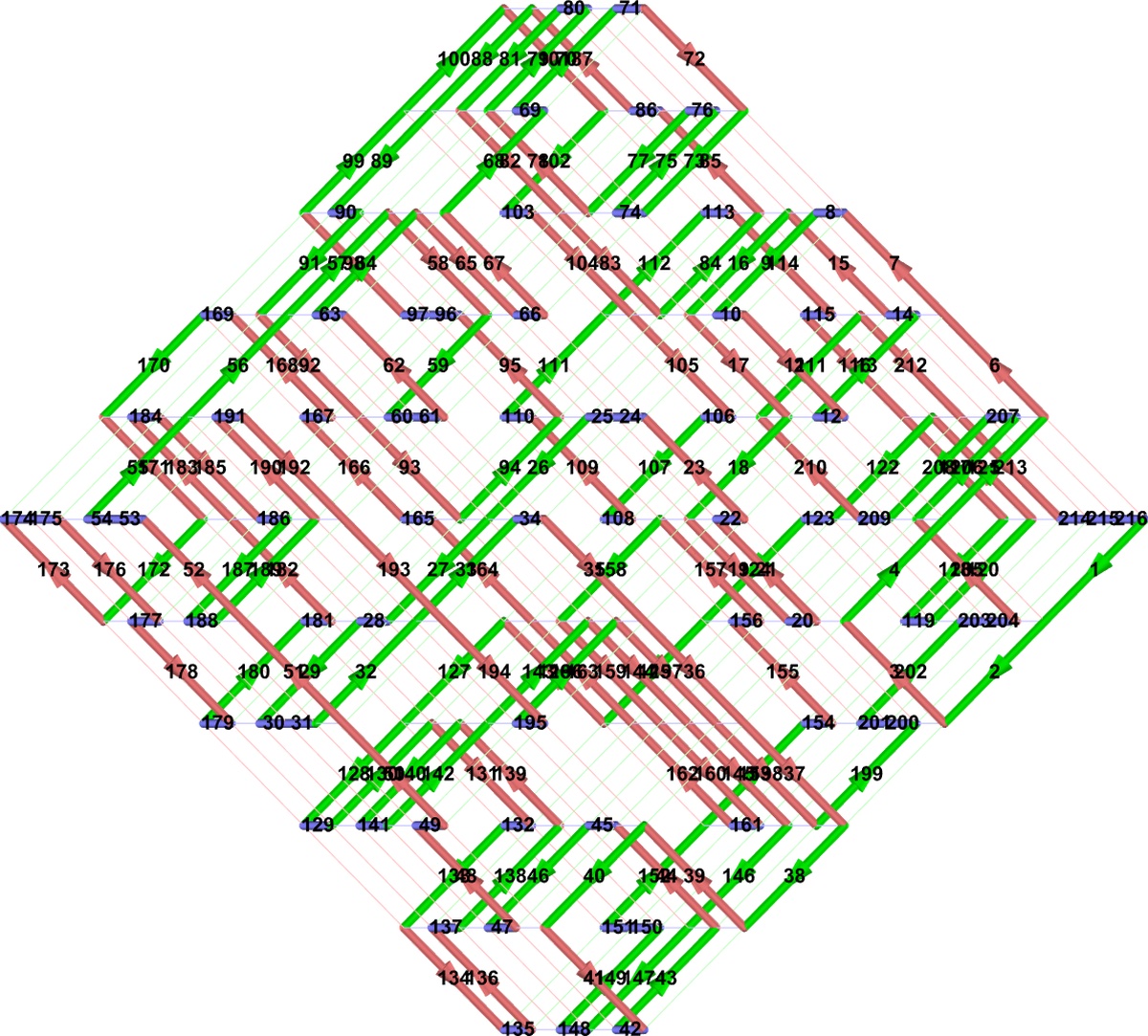}
\caption{Two examples of  grid knots. The left grid knot has  $L=3$ and 64 labeled arcs, and the right has $L=5$ and 216 labeled arcs.}
\label{fig:grid_Knot}
\end{figure}

\begin{figure}[htp]
\centering
\begin{picture}(300,140)
\put(-65,0){\includegraphics[width=.27\textwidth]{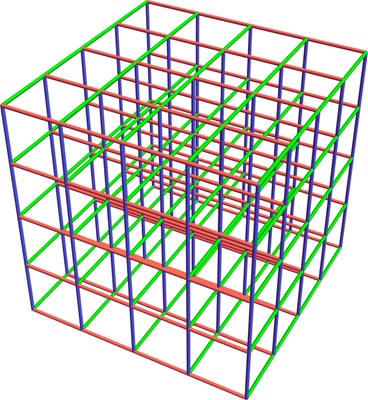}\hspace{.6cm}}
\put(80,0){\includegraphics[width=.3\textwidth]{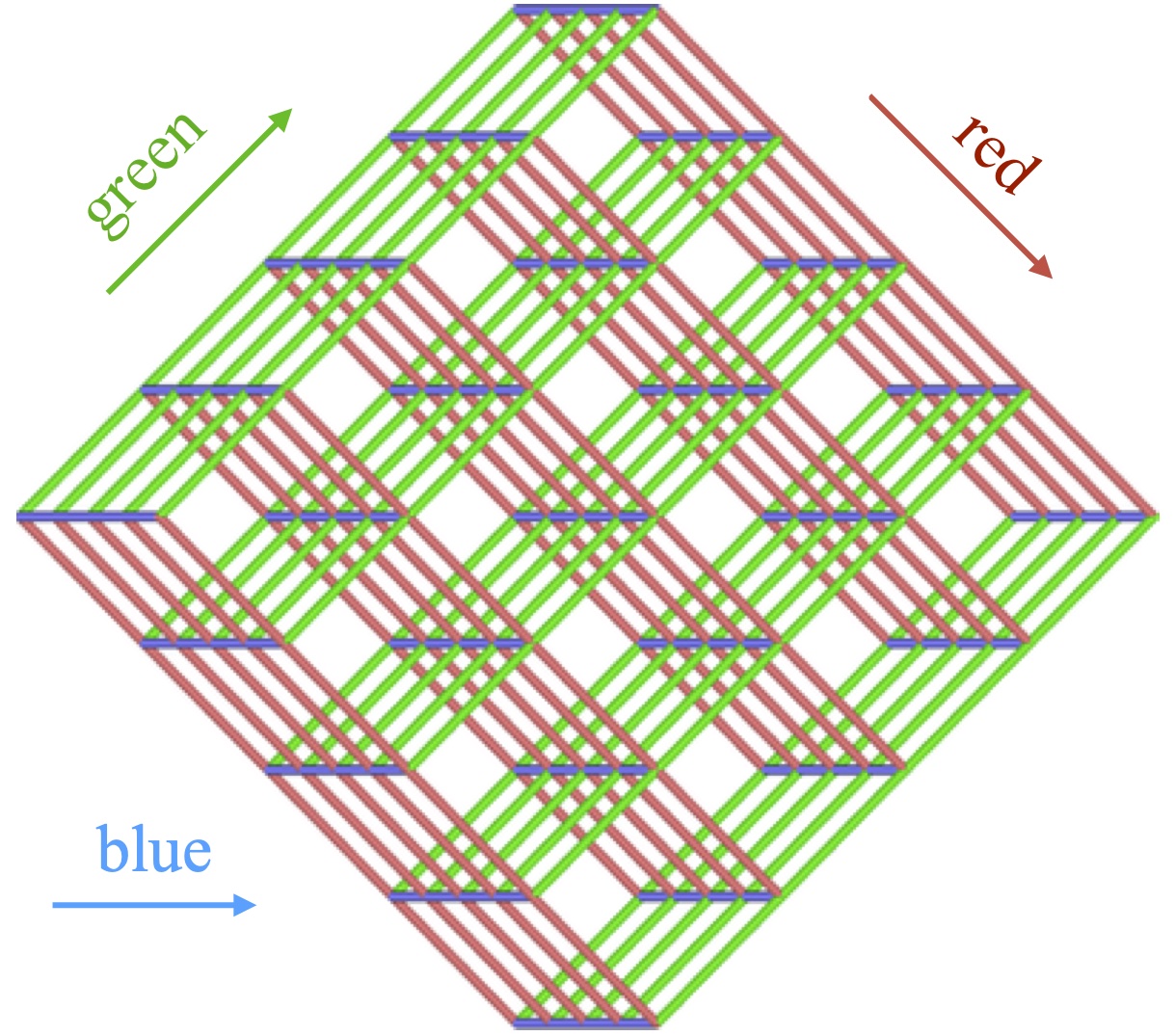}\hspace{.7cm}}
\put(250,0){\includegraphics[width=.25\textwidth]{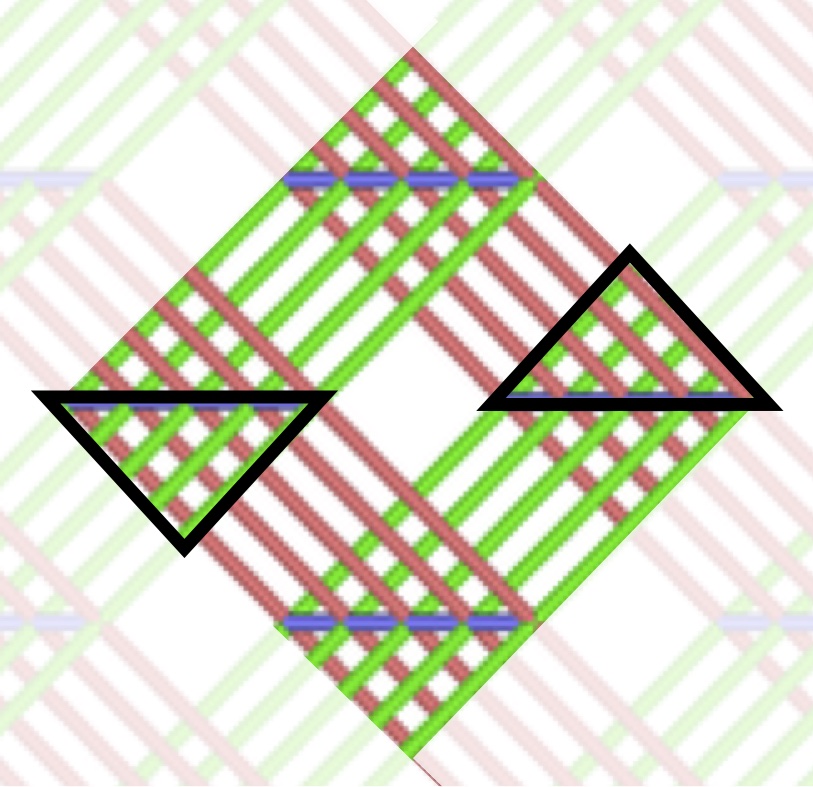}}
\put(-60,0){(A)}
\put(80,0){(B)}
\put(80,0){(B)}
\put(220,0){(C)}
\put(255,30){$F_1$}
\put(345,65){$F_2$}
\end{picture}
\caption{The grid in (B) shows a slightly askew top down view of the grid from (A). (C) highlights two crossing fields, $F_1$ and $F_2$, of a grid.}
\label{fig:grids}

\end{figure}

 The process of converting an oriented yarn ball knot of length/volume $V$ to a grid knot is as follows. 
 Replace the yarn by an approximation along grid lines with grid spacing say $\frac{1}{10}$'th the unit width of the yarn.  
 Rescale so that the grid squares are unit length again. 
 Starting at any corner of the grid knot, label the arcs of the knot in order according to the orientation.
 The resulting knot is bounded in a box of size $\sim 10^3V$, and this process takes  $\sim V$ computation steps.
To convert a grid knot to a yarn ball knot, scale the grid so the distance between neighbouring grid points is say 3 or 5 units. 
Replace the arcs of the knot with yarn of width 1 and round out the corners. 
This process takes time proportional to the length of the knot.
 When computing an invariant of a yarn ball knot, first converting the knot into a grid knot adds a negligible amount of computation time.

  For the remainder of this paper, we conventionally view grids with the slightly askew top-down view  as in Figure \ref{fig:grids} (B). From this perspective, all of the crossings of a grid knot occur in triangular \emph{crossing fields} of the grid-- highlighted in Figure \ref{fig:grids} (C). 
   The grid lines in the $x,y$ directions are colored in green and red, and the grid lines in the vertical direction are colored blue. 
From the askew top down perspective as in Figure \ref{fig:grids} (B), we keep the convention that $``/"$ grid lines are called ``green", ``\textbackslash" grid lines are called ``red", and horizontal grid lines are called ``blue".
  All of the crossings in a crossing field occur between green and red grid lines. 
The vertical blue grid lines never participate in a crossing.

There are $2L^2$ triangular crossing fields; 2 at each of the $(L-1)^2$ interior corners, and one along each exterior corner except two corners, which is $2L+2(L-1)$, for a total of $2(L-1)^2+2L+2(L-1)=2L^2$.

\subsection{Linking Number}
For a two-component link $\mathcal{L}$, the \emph{linking number} of $\mathcal{L}$, denoted $lk(\mathcal{L})$, is a classical link invariant that measures how the two components are linked. 
From a planar projection of $\mathcal{L}$, $lk(\mathcal{L})$ can be computed as follows: Only counting ``mixed'' crossings that involve both components (the over strand is from one component and the under strand is from the other component), $lk(\mathcal{L})$  is one half the difference of the number of positive crossings and the number of negative crossings. 
Using this 2D method for a planar diagram with $n$ crossings, computing $lk$ requires $\sim n$ steps-- one for every crossing\footnote{Notice that in the worst case, and presumably also typically, the number of mixed crossings is $\sim n$.}-- which shows that $C_{lk}(2D,n)\sim n$.
Using grid links, we provide a 3D algorithm that computes $lk$ in time $\sim V$, and since $n=V^{\frac{4}{3}}\gg V$, this proves that $lk$ is \emph{C3D}.

\begin{thm}\label{thm:lk}
 $C_{lk}(3D,V)\sim V$, while $C_{lk}(2D,n)\sim n$, so $lk$ is C3D.

\end{thm}

\begin{proof} It is clear that $C_{lk}(3D,V)\gtrsim V$ and $C_{lk}(2D,n)\gtrsim n$: if that wasn't the case, it would mean that $lk$ could be computed without looking at parts of the yarn (in the 3D case, for the length of the yarn is $\sim V$) or at some of the crossings (in the 2D case, for there are $n$ crossings). But this is absurd: changing any crossing could change the linking number, and likewise, slightly moving any piece of yarn. Also, the standard ``sum of signs over crossings'' formula for $lk$ shows that $C_{lk}(2D,n)\lesssim n$.

So we only need to prove that $C_{lk}(3D,V)\lesssim V$. Let $\mathcal{L}$ be a grid link with size $L$ and volume $V=L^3$. There are 8 possible ``mixed'' crossing types in $\mathcal{L}$ according to the orientations of the two strands, whether the green or red strand crosses on top.

    \begin{figure}[h]
      \centering
        \includegraphics[width=.5\textwidth]{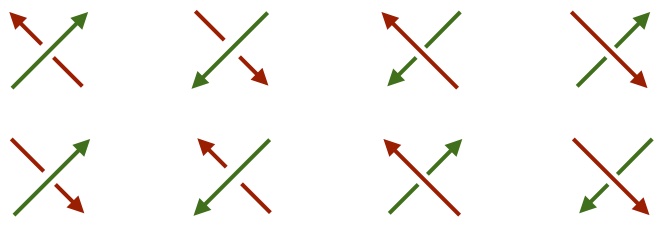}
      \caption{The 8 crossing types, where ``/" strands are green and ``\textbackslash" are red. For later use we label these crossings $x_1$ through $x_8$. }
      \label{fig:crossingTypes}
  \end{figure}
 
 The linking number is computed by counting how many crossings of each of the 8 types are in the grid knot, adding these numbers together (with signs) and dividing by 2.
 We further break this down by counting how many crossings of each type occur in a single crossing field, then sum over all the crossing fields.
 
We show here how to count instances of the first crossing listed in the diagram above where we assume the green strand comes from the link component labeled 1, and the red strand from the link component labeled 2.

The other seven cases are counted similarly and with the same computational complexity.

Fix a crossing field $F_k$ and let $G$ be a subset of the green arcs of the grid link labeled $(1,p)$, for any $p$, in $F_k$ and let $R$ be a subset of the red edges of the grid link in $F_k$ labeled $(2,p)$, for any $p$.
For any such crossing field $F_k$, we can define a height map on the sets $R$ and $G$, denoted $z:R,G \rightarrow [0,L]$, which gives the vertical strand height in $F_k$. 

Within the fixed $F_k$, to count the number of times a green strand crossed over a red strand, we need to count the number of pairs $\{(r,g)\in R\times G: z(r)<z(g)\}$. 
To do this, line up the elements of $R$ and $G$ in increasing order based on their $z$-value, as shown below.
\begin{figure}[htp]

\includegraphics[width=.5\textwidth]{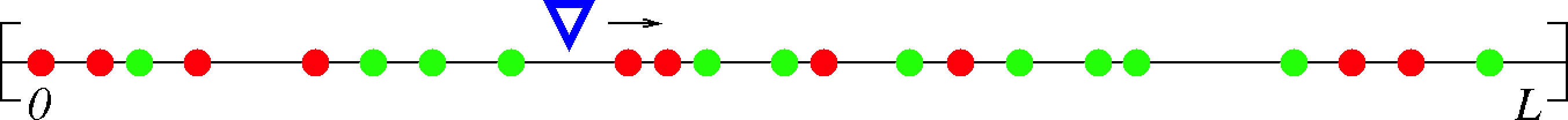}
\label{fig:gridKnot}
\end{figure}

If there is a pair $(r,g)$ of $R\times G$ such that $z(r)=z(g)$, there is a convention for which order to place the dots on the line. If we are counting greens over red, put the green dot first, and red dot second, as shown in Figure \ref{fig:Delta_comp}. If we count red crossing over green, put the red dot first and green dot second.
Start with $rb=cf=0$ (``reds before" and ``cases found") and with the place holder $\nabla$ before the leftmost dot. Slide $\nabla$ from left to right, incrementing $rb$ by one each time $\nabla$ crosses a red dot and incrementing $cf$ by $rb$ each time $\nabla$ crosses a green dot. The value of $cf$ is the desired number of pairs that we wished to count. An example of the computation is shown in Figure \ref{fig:Delta_comp}.

\begin{figure}[ht]
\includegraphics[width=.8\textwidth]{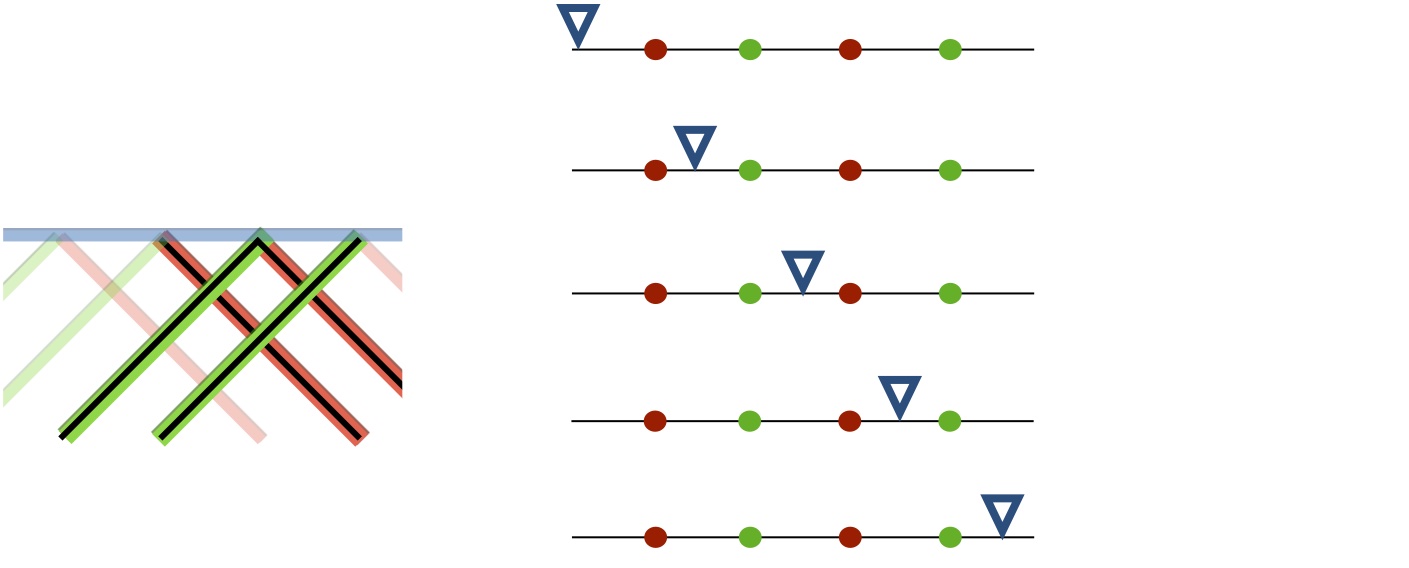}
\put(-340,89){1}
\put(-315,89){2}
\put(-290,89){3}
\put(-265,89){4}
\put(-192,118){2}
\put(-192,88){2}
\put(-192,58){2}
\put(-192,25){2}
\put(-192,-5){2}
\put(-169,118){3}
\put(-169,88){3}
\put(-169,58){3}
\put(-169,25){3}
\put(-169,-5){3}
\put(-142,118){3}%
\put(-142,88){3}
\put(-142,58){3}
\put(-142,25){3}
\put(-142,-5){3}
\put(-119,118){4}%
\put(-119,88){4}
\put(-119,58){4}
\put(-119,25){4}
\put(-119,-5){4}
\put(-85,125){$rb=cf=0$}
\put(-85,95){$rb=1$, $cf=0$}
\put(-85,65){$rb=1$, $cf=1$}
\put(-85,35){$rb=2$, $cf=1$}
\put(-85,5){$rb=2$, $cf=3$}
\caption{On the left is a portion of a grid knot passing through a crossing field where the green strands cross over the red strands. On the right, we show an example of the algorithm to count the number of crossings in this crossing field.}
\label{fig:Delta_comp}
\end{figure}

For a fixed crossing field $F_k$, this computation can be carried out in time $\sim L$. There are $\sim L^2$ crossing fields and 8 types of colored orientation crossing types, which yields an overall computation time of $\sim L^3=V$.
\end{proof}

\section{Finite Type Invariants}\label{sec:FiniteType}

A \emph{Gauss diagram} of an $n$-crossing knot diagram parameterized by an interval $I$ is given by the interval $I$ along with $n$ arrows. Each arrow corresponds to one of the $n$ crossings and connects a point in the parameter space $I$ which parametrizes the top of the crossing to a point which parametrizes the bottom of the crossing. Each arrow is also decorated with a sign corresponding to the sign of the crossing. We give an example of a Gauss diagram in Figure \ref{fig:gaussdiagrams}.

\begin{figure}[htp]

\centering
\begin{picture}(300,120)
\put(20,70){\includegraphics[width=.55\textwidth]{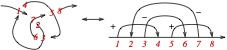}\hspace{.6cm}}
\put(20,0){\includegraphics[width=.55\textwidth]{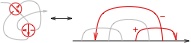}\hspace{.7cm}}

\put(0,70){(A)}
\put(0,0){(B)}

\end{picture}
\caption{(A) An example of the Gauss diagram of a tangle diagram. \\ (B) A $2$-arrow subdiagram of a Gauss diagram.}
\label{fig:gaussdiagrams}

\end{figure}

A \emph{$d$-arrow subdiagram} of a Gauss diagram $D$ is a Gauss diagram consisting of a subset of $d$ arrows from $D$.  
This subdiagram corresponds to a choice of $d$ crossings in the knot diagram represented by the Gauss diagram. An example is shown in Figure \ref{fig:gaussdiagrams}.
The space $GD=\langle $Gauss diagrams$\rangle$ is the $\Q$-vector space of all formal linear combinations of Gauss diagrams. We will denote the subspace of Gauss diagrams with $d$ or fewer arrows by $GD_d=\langle$Gauss diagrams$\rangle_d$. 

Gauss diagrams play an important role in the theory of finite type, or Vassiliev, invariants \cite{Vassiliev90, Vassiliev}. Any knot invariant $V$ taking numerical values can be extended to an invariant of knots with finitely many double points (i.e. immersed circles whose only singularities are transverse self-intersections), using the following locally defined equation:
$$V(\doublepoint)=V(\overcrossing) - V(\undercrossing)$$
The above should be interpreted in the setting of knot diagrams which coincide outside of the given crossing.
A knot invariant is a \textit{finite type invariant of type $d$} if it vanishes on all knot diagrams with at least $d+1$ double points \cite{BL93,  BN1}.

Let $\varphi_d:\{$knot diagrams$\}\rightarrow GD_d$ be the map that sends a knot diagram to the sum of all of the subdiagrams of its Gauss diagram which have at most $d$ arrows.
A Gauss diagram with $n$ arrows has $\sum_{i=1}^d$ $n\choose i$ subdiagrams with $d$ or fewer arrows. Because $n\choose i $ $\sim n^i$, a Gauss diagram with $n$ arrows has $\sum_{i=1}^d$ ${n\choose i}\sim \sum_{i=1}^d n^i\sim n^d$ subdiagrams with $d$ or fewer arrows. So $\varphi_d$ evaluated at a knot diagram with $n$ crossings will be a sum of $\sim n^d$ subdiagrams.
The map $\varphi_d$ is \emph{not} an invariant of knots, but, every finite type invariant factors through $\varphi_d$, as described in the following theorem.

\begin{thm}[Goussarov-Polyak-Viro \cite{GPV}, see also \cite{Roukema}]\label{thm:FactorsThruPhi} 
A knot invariant $\zeta$ is of type $d$ if and only if there is a linear functional $\omega$ on $GD_d$ such that $\zeta=\omega\circ\varphi_d$.
\end{thm}

A corollary of Theorem \ref{thm:FactorsThruPhi} is that any type $d$ invariant can be computed from an $n$-crossing  planar diagram $D$ in the time that it takes to inspect all $\binom{n}{d}$ size $d$ subdiagrams of $D$ for the purpose of computing $\varphi_d$:

\begin{thm} \label{Thm:FiniteType2} (see also \cite{BN2}) If $\zeta$ is a finite type invariant of type $d$ then $C_\zeta(2D, n)$ is at most $\sim n^d$.\qed
\end{thm}

Next, we prove that in fact finite type invariants can be computed more efficiently from a 3D presentation:

\begin{thm}\label{Thm:FiniteTypeD}
If $\zeta$ is a finite type invariant of type $d$ then $C_\zeta(3D, V)$ is at most ~$V^d$.
\end{thm}

\begin{proof}

Let $\zeta$ be a finite type invariant of type $d$ and let $K$ be a grid knot with side lengths $L$ viewed as a diagram from the top-down perspective as in Figure \ref{fig:grids}.
By Theorem \ref{thm:FactorsThruPhi},  $\zeta(K)=\omega\circ\varphi_d(K)$, for some linear functional $\omega$.  
Since $GD _d$ is a fixed finite-dimensional vector space, the complexity of computing $\omega$ does not depend on $K$ or $V$. Thus, to prove the theorem, it suffices to show that $\varphi_d(K)$ can be computed in time $V^d$.

  By definition, $\varphi_d(K)=\sum c_D D$ where $D$ ranges over all  possible Gauss diagrams with at most $d$ arrows, and $c_D$ is the number of times $D$ occurs as a subdiagram in the Gauss diagram for $K$. 
  The diagram $D$ has arrows decorated with a $\pm$ sign corresponding to $\pm$ crossings in $K$. 
  In a grid knot, there are 8 different realizations of oriented crossings depending on the colorings of the strands, which are shown in Figure \ref{fig:crossingTypes}. 

 To compute the coefficient $c_D$, we need to count the number of times $D$ occurs as a subdiagram of $K$. Since $K$ is a grid knot, we will subdivide this count by further specifying what type of $\pm$ crossing is associated to each $\pm$ arrow in $D$. 
 To do this, we need to consider a more detailed labeling of Gauss diagrams.
 Let $LGD_d=\langle$Labeled Gauss Diagrams$\rangle_d$ be the space of Gauss diagrams with at most $d$ arrows where each arrow is decorated with a label in $\{x_1,\cdots, x_8\}$.
 These labels will denote the \emph{crossing type of an arrow}.
  $LGD_d$ is large but finite dimensional, and for each diagram $D\in GD_d$ with $\ell$ arrows, there are $8^\ell$ related diagrams $D_\chi \in LGD_d$ where $\chi \in\{x_1,\cdots, x_8\}^{\ell}$ is a sequence specifying the labelling of the arrows of $D_\chi $. 
  Thus, $c_D$ can be computed by 
  $$c_D=\sum_{\chi }c_{D_\chi }$$ where $\chi $ ranges over all sequences in $\{x_1,\cdots x_8\}^{\ell}$. The coefficient $c_{D_\chi }$ is the number of times $D_\chi $ occurs as a subdiagram of $K$ where an arrow $j$ in $D_\chi $ with label $\chi _j$ corresponds to a crossing in $K$ of type $\chi _j$.

 To compute $\varphi_d(K)$ it suffices to compute $c_{D_\chi }$ for every $D_\chi \in LGD_d$. The following argument will compute $c_{D_\chi }$ and will be repeated for every $D_\chi \in LGD_d$. Such repetition will not contribute to the complexity of $\varphi_d$ up to $\sim$ as $LGD_d$ is a fixed finite dimensional space independent of $K$.
  
Let $D_\chi \in LGD_d$ be a labeled Gauss diagram with  $\ell\leq d$ arrows. The most computationally difficult case is when $\ell=d$, so we will count instances of diagrams with $\ell=d$ and all other counts will be similar and easier. We count all instances of $D_\chi $ that fall into specific crossing fields of the grid knot $K$. 

Number each arrow of $D_\chi $ with $j\in\{1,\cdots, d\}$, in the order in which the arrows first
occur from left to right in $D_\chi $.
As a labeled Gauss diagram, each arrow $j$ of $D_\chi $ is also decorated with $\chi _j\in \{x_1,\cdots, x_8\}$ which specifies the crossing type associated to arrow $j$. 
We also label the ends of the arrows. 
Following the example in Figure \ref{fig:labeledGauss}, label the head of arrow $j$ with $\alpha (j)\in \{1,\cdots, 2d\}$ and each tail by $\beta(j)\in \{1,\cdots, 2d\}$ in increasing order according to the parametrization. We will count the occurrences of $D_\chi $ in $K$ by scanning the grid knot in the following way:

1. We consider all possible $d$-tuples $\overline{F}=(F_{k_1},\hdots,F_{k_d})$ of crossing fields in $K$. For such a tuple, the $j$th arrow in $D_\chi $ gets assigned the crossing field denoted $F_{k_j}$. This choice specifies that the arrow $j$ corresponds to a crossing of type $\chi_j$ which occurs in the crossing field $F_{k_j}$ in the grid knot diagram for $K$ (if such a crossing exists in $F_{k_j}$, see Example \ref{ex:Buckets}). 
There are $\sim L^2$ choices of crossing fields for each of the $d$ arrows, so there  $\sim L^{2d}$ possible choices of crossing field tuples for $D_\chi$.

2. For any $d$-tuple $\overline{F}$ of crossing fields, we also assign to each arrow of $D_\chi$ a pair of sets: $B_i$, associated to the head, and $B_{i'}$ associated to the tail of the arrow. The index $i$ of $B_i$ will not necessarily match the label $j$ of the arrow, but rather the labelling of the $B_i$'s is according to the order in which the ends of the arrows appear along the parametrization, as in Figure \ref{fig:labeledGauss}.
Thus, the set $B_i$ associated to the head of arrow $j$ will have $i=\alpha(j)$, and the $B_{i'}$ associated to the tail of arrow $j$ will have $i'=\beta(j)$.

\begin{figure}[htp]
\begin{picture}(300,70)
\put(30,10){\includegraphics[width=.5\textwidth]{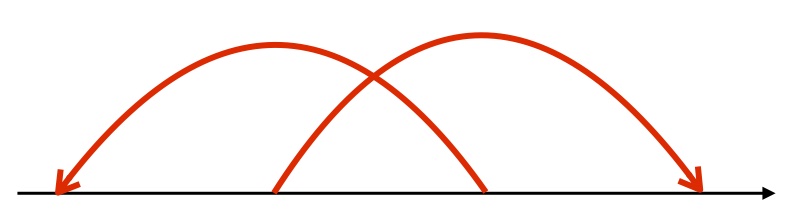}\hspace{.6cm}}
\put(20,7){$\alpha(1)=1$}
\put(40,-5){$B_1$}
\put(80,7){$\beta(2)=2$}
\put(100,-5){$B_2$}
\put(100,65){$\chi_1$}
\put(160,65){$\chi_2$}
\put(140,7){$\beta(1)=3$}
\put(160,-5){$B_3$}
\put(200,7){$\alpha(2)=4$}
\put(220,-5){$B_4$}
\put(45,40){$F_{k_1}$}
\put(210,40){$F_{k_2}$}
\end{picture}
\caption{An example of a labeled Gauss diagram with $d=2$.}
\label{fig:labeledGauss}
\end{figure}

 The sets $B_i$ and $B_{i'}$ for arrow $j$ will depend on both the crossing field $F_{k_j}$ and the crossing type $\chi_j$ of the arrow. 
 The crossing type $\chi_j$ determines whether we count green strands on top of red, or red on top of green, and which orientation of the strands to count.
 In general, $B_i$ will be the set of strands of the knot in $F_{k_j}$ that have the same color and orientation as the under strand in crossing type $\chi_j$. The set $B_{i'}$ will be the strands of the knot in $F_{k_j}$ that have the same color and orientation as the over strands in crossing type $\chi_j$.

\begin{figure}[h]
    \centering
    \includegraphics[width=.55\textwidth]{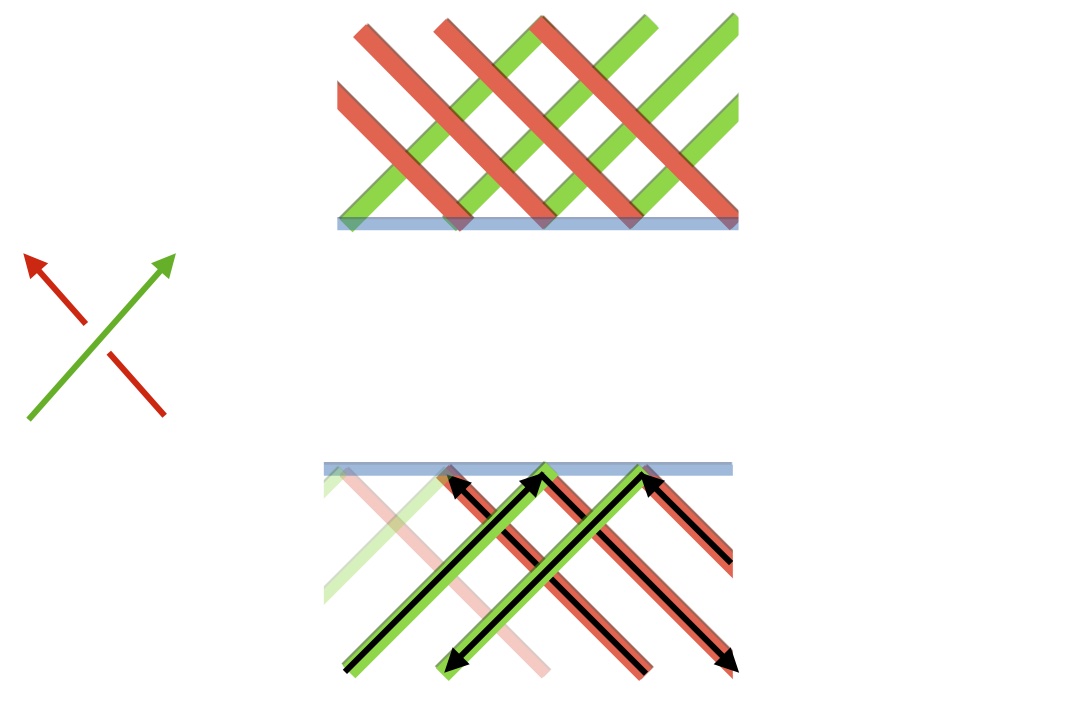}
    \put(-225,57){$x_1$}
    \put(-60, 130){$\longrightarrow$}
    \put(-60, 40){$\longrightarrow$}
    \put(-175, 165){$F_{k_j}$ with red on top}
    \put(-175, 65){$F_{k_j}$ with green on top}
    \put(-172,3){$g_1$}
    \put(-150,3){$g_2$}
    \put(-103,3){$r_1$}
    \put(-77,10){$r_2$}
    \put(-77,32){$r_3$}
    \put(-30, 140){$B_i=\emptyset$}
    \put(-30,120){$B_{i'}=\emptyset$}
    \put(-30, 50){$B_i=\{r_1, r_3\}$}
    \put(-30,30){$B_{i'}=\{g_1\}$}
    \caption{An example of computing the sets $B_i$ and $B_{i'}$ for a crossing of type $\chi_i=x_1$ and two different crossing fields.}
    \label{fig:BucketExample}
\end{figure}

\begin{ex}\label{ex:Buckets}
Suppose $\chi_j=x_1$ and refer to Figure \ref{fig:BucketExample}. The crossing type $x_1$ determines we are counting green and red arrows both oriented upwards, with green on top.
We look to the crossing field $F_{k_j}$ of our knot diagram for $K$.
In $F_{k_j}$, if all of the red strands cross on top, then both $B_i$ and $B_{i'}$ are empty, as the crossings are not compatible with the $x_1$ type crossing.
If the green strands in $F_{k_j}$ cross on top,
then the set $B_i$, associated to the head of arrow $j$,  will be all of the red strands of the knot in $F_{k_j}$ oriented upwards.
The set $B_{i'}$, associated to the tail of arrow $j$, will be all of the green strands of the knot in $F_{k_j}$ oriented upwards.
\end{ex}

3. With the crossing fields chosen and the $B_i$'s assigned, we look for instances of $D_\chi$ within the grid knot, by choosing a strand $b_i$ in each $B_i$ and checking if the collection of chosen strands $\{b_1,\cdots, b_{2d}\}$ is compatible with the parametrization and gives rise to exactly $d$ crossings in $K$ that recover the diagram $D_\chi$. To check this compatibility, there are two functions defined on the sets $B_i$.  \begin{itemize}
    \item $t:\cup B_i\rightarrow \mathbb{Z}$ gives the order in which the arrow endpoints labeled by the $B_i$'s occur in the parametrization of the grid diagram for $K$.
    \item $z:\cup B_i\rightarrow \{0,\cdots,L\}$ gives the vertical height of the strand in the grid knot. Notice that  for each $i$, the elements of $B_i$ are strands of the same color inside a single crossing field, all of which have distinct heights. So $z|_{B_i}$ is injective for every $i$.
\end{itemize}

First, we need the crossings to occur in the correct order along the grid knot according to the parametrization. This means we need $t(b_1)<t(b_2)<...<t(b_{2d})$.
Second, we need to ensure that the chosen strands for a given arrow in the diagram actually cross. For an arrow $j$, $b_{\beta(j)}$ needs to be the over strand of the crossing and $b_{\alpha({j})}$ is the under strand. To get that $b_{\beta({j})}$ crosses over $b_{\alpha(j)}$, we need the height of $b_{\alpha(j)}$ to be lower than the height of $b_{\beta(j)}$, i.e. $z(b_{\alpha(j)})<z(b_{\beta(j)})$.
Notice that $b_{\alpha(j)}$ and $b_{\beta(j)}$ can be elements from different $B_i$'s.

 This problem reduces to the following counting problem:
Given  $\alpha(j), \beta(j)\in \{1,\cdots, 2d\}$ for $j\in \{1,\cdots, d\}$, the $2d$ sets $B_i$ with $i\in \{1,\cdots, 2d\}$ and functions $t:\cup B_i\rightarrow\mathbb{Z}$ and $z:\cup B_i\rightarrow\mathbb\{0,\cdots, L\}$ such that $z|_{B_i}$ and $t$ are injective, we want to compute $|A|$ where 

\[
  A=\left \{
  \raisebox{0mm}{$b\in(b_i)^{2d}_{i=1}\in\prod_i B_i$}
  \left|
  \parbox{2.25in}{\centering $t(b_1)<t(b_2)<\cdots <t(b_{2d})$,\\  $\forall j\in \{1,\cdots,d\}, z(b_{\alpha(j)})<z(b_{\beta(j)})$ }
  \right.\right\}.
\]

\noindent In Section \ref{sec:comb} we will prove Proposition \ref{Prop:CountingL^d}, which asserts that this computation can be carried out in time ~$L^d$.

Since we had $\sim L^{2d}$ choices of $d$-tuples of crossing fields, a large but finite constant number of labeled Gauss diagrams with $d$ or fewer arrows, and for each choice we have $L^d$ computations to find instances of this diagram in the knot, we get a total computation time of $\sim L^{2d}L^d=V^d$ as claimed.
\end{proof}

\section{Combinatorial Results}\label{sec:comb}

In this section, we prove Proposition \ref{Prop:CountingL^d}, which was used in the proof of Theorem \ref{Thm:FiniteTypeD}. We start with the following lemma.

\begin{lem}\label{lem:JustTRelations}
Suppose we have sets $B_i$, for $i\in\{1,\cdots,2d\}$, and a map $t:\cup B_i\rightarrow \mathbb{N}$ such that $t|_{B_i}$ is injective for all $i$. Let $K\coloneqq\max(|B_i|)$. Then the quantity
\[
N=\left|\left\{ b=(b_i)^{2d}_{i=1}\in\prod^{2d}_{i=1} B_i: t(b_1)<t(b_2)<\cdots<t(b_{2d})\right\}\right|
\]
can be computed in time $\sim K$.
\end{lem}

\begin{proof}
In time $\sim K$ each of the $B_i$'s can be sorted by the values of $t$ on it and replaced by its indexing interval. So without loss of generality, each $B_i$ is just a list of integers $\{1,\ldots,|B_i|\}$, the function $t$ is replaced by increasing functions $t_i\colon\{1,\ldots,|B_i|\}\to{\mathbb N}$, and we need to count
\[
N=\left|\left\{ b=(b_i)^{2d}_{i=1}\in\prod^{2d}_{i=1} \{1,\ldots,|B_i|\}: t_1(b_1)<t_2(b_2)<\cdots<t_{2d}(b_{2d})\right\}\right|
\]
For $1\leq\iota\leq 2d$ and $1\leq\tau\leq|B_\iota|$, let
\[N_{\iota,\tau}=\left|\left\{ b=(b_i)^{\iota}_{i=1}\in\prod_{i=1}^\iota \{1,\ldots,|B_i|\}: t_1(b_1)<t_2(b_2)<\cdots<t_\iota(b_{\iota})\leq t_\iota(\tau)\right\}\right|,\]
i.e. $N_{\iota,\tau}$ is the number of sequences of length $\iota$ with increasing $t$-values that terminate at a value less than or equal to $t_\iota(\tau)$. Also set $N_{\iota,0}=0$ for all $\iota$. Then clearly $N=N_{2d,|B_{2d}|}$, $N_{1,\tau}=\tau$ for all $\tau$ and for $\iota>1$,
\[ N_{\iota,\tau}=N_{\iota,\tau-1} + N_{\iota-1,\tau'}, \]
where $\tau'=\max(\{0\}\cup\{\tau''\colon t_{\iota-1}(\tau'')<t_\iota(\tau)\})$. Note that $\tau'$ can be computed in time $\log K\sim 1$ and hence the $N_{\iota,-}$'s can be computed from the $N_{\iota-1,-}$'s in time $\sim K$. To find the $N_{2d,-}$'s and hence $N$ this process needs to be repeated $2d\sim 1$ times, and the overall computation time remains $\sim K$.
\end{proof}

\begin{prop}\label{Prop:CountingL^d}
Given a collection of $2d$ sets $B_i$ with $i\in \{1,\cdots, 2d\}$, and functions $\alpha, \beta: \{1,\cdots, d\} \rightarrow\{1,\cdots, 2d\}$, $t:\cup B_i\rightarrow\mathbb{Z}$ and $z:\cup B_i\rightarrow\{0,\cdots, L\}$ such that $im{(\alpha)}\cup im{(\beta)}=\{1,\hdots,2d\}$, and $z|_{B_i}$ and $t|_{B_i}$ are injective, the size of the following set can be computed in time $\sim L^d$, 
\[
  A=\left\{
  \raisebox{0mm}{$b\in(b_i)^{2d}_{i=1}\in\prod B_i$}
  \left|
  \parbox{2.5in}{\centering $t(b_1)<t(b_2)<\cdots <t(b_{2d}),$\\  $\forall j\in \{1,\cdots,d\}, \; z(b_{\alpha(j)})<z(b_{\beta(j)})$ } \right.
  \right\}.
\]

\end{prop}

\begin{proof}  Lemma \ref{lem:JustTRelations} shows us how to count elements in the set $A$ without the conditions on $z$. We will show that $|A|$ can be computed by writing $A$ as a union of sets with the $t$-conditions and \emph{one Cartesian condition}, so that we can apply  Lemma \ref{lem:JustTRelations} to count elements in each set of the union.

Let $p\in \mathbb{N}$ so that $2^{p-1}<L\leq2^p$. Every $z$-value is in $\{0, \cdots, L\}$ and can be written as a binary expansion with exactly $p$ binary digits, padding with zeros in front if needed. 
For example, if $p=5$ the number 2 can be written as 00010.
For two binary numbers $z_1$ and $z_2$ in this form,  $z_1< z_2$ if they have the same binary expansions from left to right up to a point, and in the first place they differ $z_1$ has a 0 and $z_2$ has a 1. The notation we use to describe this is as follows. 
Let $\sigma$ be the binary sequence from left to right which $z_1$ and $z_2$ share in common. We write $z_1=\sigma 0*$ to mean the binary expansion of $z_1$ read left to right is $\sigma$ followed by 0 followed by an arbitrary remainder of 0's and 1's. Similarly, we write $z_2=\sigma 1 *$.

For each $j\in \{1,\cdots,d\}$, if the condition $z(b_{\alpha(j)})<z(b_{\beta(j)})$ is satisfied, there exists a binary sequence $\sigma_j$ of length $|\sigma_j|<p$ so that $z(b_{\alpha(j)})=\sigma_j 0*$ and $z(b_{\beta(j)})=\sigma_j 1*$. 
This binary structure is key to construct the desired \emph{Cartesian} conditions to compute $|A|$.

For each $j\in \{1,\cdots,d\}$, the relation $z(b_{\alpha(j)})<z(b_{\beta(j)})$ requires the pair $(z(b_{\alpha(j)}),z(b_{\beta(j)}))$ to be in the triangle below the diagonal shown in Figure \ref{fig:approximatingTirangle}.
Notice that in a single pair, $(z(b_{\alpha(j)}),z(b_{\beta(j)}))$, the $b_{\alpha(j)}$ and $b_{\beta(j)}$ are not necessarily elements of the same $B_i$.
Since the outputs of $z$ are integral, we can divide the triangle into a finite number of squares below the diagonal. 
The sides of the squares in the triangle are labeled by binary sequences of length less than or equal to $ p$. Each square is determined by all but the last entry in the sequences labeled on the right and bottom sides of the square.

\begin{figure}[htp]
\begin{picture}(100,175)
\put(-50,0){\includegraphics[width=.5\textwidth]{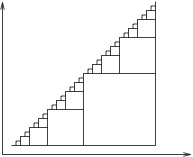}}
\put(50,85){$q=0$}
\put(10,45){$q=1$}
\put(90,125){$q=1$}
\put(150,10){$z(b_{\beta(j)})$}
\put(-40,170){$z(b_{\alpha(j)})$}

\put(125,-12){$2^p$}
\put(125,0){$|$}
\put(85,15){\textcolor{red}{1}}
\put(120,50){\textcolor{red}{0}}

\put(115,110){\textcolor{green}{1}\textcolor{red}{0}}
\put(100,96){\textcolor{green}{1}\textcolor{red}{1}}

\put(71,96){\textcolor{green}{\tiny 10}\textcolor{red}{\tiny 1}}
\put(73,108){\textcolor{green}{\tiny 10}\textcolor{red}{\tiny 0}}

\put(114,149){\textcolor{green}{\tiny 11}\textcolor{red}{\tiny 0}}
\put(110,137){\textcolor{green}{\tiny 11}\textcolor{red}{\tiny1}}

\put(33,33){\textcolor{green}{0}\textcolor{red}{0}}
\put(22,15){\textcolor{green}{0}\textcolor{red}{1}}

\put(33,66){\textcolor{green}{\tiny01}\textcolor{red}{ \tiny 0}}
\put(28,56){\textcolor{green}{ \tiny 01}\textcolor{red}{\tiny 1}}

\put(-9,25){\textcolor{green}{\tiny 00}\textcolor{red}{\tiny 0}}
\put(-13,15){\textcolor{green}{\tiny 00}\textcolor{red}{\tiny 1}}
\end{picture}
\caption{The sides of each square are labeled by  binary sequences which describe the set of integral pairs within the square. For example, the square with south side labeled by \textcolor{green}{1}\textcolor{red}{1} and east side labeled by \textcolor{green}{1}\textcolor{red}{0} contains pairs of the form $(z(b_{\alpha(j)}), z(b_{\beta(j)}))$ with $z(b_{\alpha(j)})=$\textcolor{green}{1}$\cdot 2^{p-1}+$\textcolor{red}{0}$\cdot 2^{p-2}+\cdots$ and $z(b_{\beta(j)})=$\textcolor{green}{1}$ \cdot 2^{p-1}+$\textcolor{red}{1}$\cdot 2^{p-2}+\cdots$.}
\label{fig:approximatingTirangle}
\end{figure}

Since $j$ can range from $1$ to $d$, there are $d$ relations $z(b_{\alpha(j)})<z(b_{\beta(j)})$ that must simultaneously hold. So, we consider $d$ copies of the above triangle, one for each $j$. We will write $A$ as a union of sets with the $t$-conditions and one Cartesian condition corresponding to a specific square in each of the triangles.

The collection of squares on a subdiagonal all have the same size and are labeled by binary sequences of the same length. To each subdiagonal, we can associate the $q$-value $$q=\text{(the length of the binary sequences on the subdiagonal)}-1.$$ 
The subdiagonals with $q$-values $0$ and $1$ are shown in Figure \ref{fig:approximatingTirangle}.

Let $\bar \sigma=(\sigma_j)_{j=1}^d$ be a $d$-tuple of binary sequences, where the length of $\sigma_j$ is $|\sigma_j|=q_j\in \{0,\cdots,p-1\}$. As demonstrated in Figure \ref{fig:triangles}, the coordinates of $\bar \sigma$ pick out a specific square in each of the $d$ triangles, ($\sigma_j$ is a binary sequence that labels a square in triangle $j$).

\begin{figure}[htp]
\begin{picture}(80,150)
\put(-40,0){\includegraphics[width=.35\textwidth]{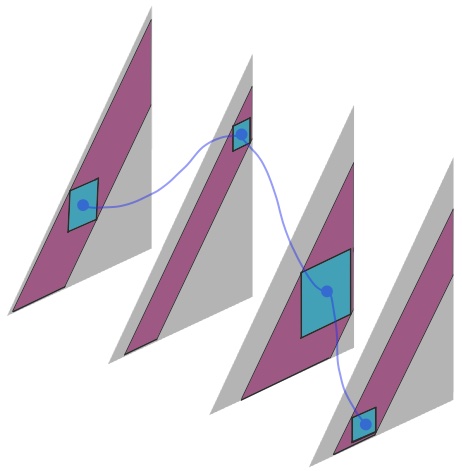}}
\put(76,0){\textcolor{teal}{$\sigma_1$}}
\put(80,65){\textcolor{teal}{$\sigma_2$}}
\put(45,110){\textcolor{teal}{$\sigma_3$}}
\put(-30,95){\textcolor{teal}{$\sigma_4$}}

\put(115,80){\textcolor{purple}{\large{$q_1$}}}
\put(80,90){\textcolor{purple}{\large{$q_2$}}}
\put(45,120){\textcolor{purple}{\large{$q_3$}}}
\put(12,140){\textcolor{purple}{\large{$q_4$}}}
\end{picture}

\caption{A schematic diagram of $\bar \sigma=(\sigma_1,\cdots, \sigma_4)$ and $\bar q=(q_1,\cdots, q_4)$. The coordinate $\sigma_j$ specifies a square in the $j$th triangle, on the subdiagonal corresponding to $|\sigma_j|=q_j$.}
\label{fig:triangles}
\end{figure}

The length of $\sigma_j$, $q_j$, is the $q$-value for the subdiagonal that contains the square labeled by $\sigma_j$. 
The coordinate $(z(b_{\alpha(j)}),z(b_{\beta(j)}))$ is in the square labeled by $\sigma_j$ if $z(b_{\alpha(j)})=\sigma_j0*$  and $z(b_{\beta(j)})=\sigma_j1*$.

For a given $\bar{\sigma}=(\sigma_j)_{j=1}^d$, we can define $A_{\bar{\sigma}}$, which collects all the choices of $2d$-tuples of $b_i$'s that satisfy the $t$-conditions and for each $j$ satisfy the $j$th $z$-condition inside the square $\sigma_j$.

\begin{equation}\label{eq:As}
A_{\bar{\sigma}}=\left\{
  \raisebox{0mm}{$\bar b=(b_i)^{2d}_{i=1}\in\prod B_i$}
  \left|
  \parbox{2.3in}{\centering $t(b_1)<t(b_2)<\cdots <t(b_{2d}),$\\ 
  $\forall j\in \{1,\cdots,d\}, z(b_{\alpha(j)})=\sigma_j 0*$\\
  \text{ and } $z(b_{\beta(j)})=\sigma_j1* $}
  \right. \right\}.
\end{equation}

If we let $ \bar q=(q_j)_{j=1}^{d}\in\{0,\cdots, p-1\}^d$, we can consider all $\bar \sigma$'s so that the $j$th coordinate in $\bar \sigma$ has length equal to the $j$th entry in $\bar q$, or $|\sigma_j|=q_j$. Then for we can define 

\begin{equation}\label{eq:Aq}
     A_{\bar q}=\bigcup_{\bar\sigma: \forall j, |\sigma_j|=q_j} A_{\bar \sigma},
\end{equation}

\noindent which collects all the choices of $b_i$ that satisfy the $t$-conditions and satisfy the $j$th $z$-condition inside the diagonal $q_j$. Finally, collecting up all possible $A_{\bar q}$'s gathers every possible desired choice of $b_i$'s that satisfy the $t$-conditions and the $z$-conditions, which is the desired set $A$,
\begin{equation}\label{eq:A}
    A=\bigcup_{\bar q \in\{0,\cdots,p-1\}^d} A_{\bar q}.
\end{equation} 

In essence, we break $A$ into pieces by specifying that the $z$-condition must be satisfied on certain diagonals, the sets $A_{\bar q}$. Then we further break down $A_{\bar q}$ by specifying that the $z$-condition must be satisfied in specific squares within the diagonals $\bar q$, the sets $A_{\bar \sigma}$.

To compute the size of these sets, first notice from Equation \ref{eq:As} that $A_{\bar \sigma}$ can be written as the set $N$ from Lemma \ref{lem:JustTRelations} with the sets $B_i$ replaced by subsets $B_i'\subseteq B_i$, where  
\[
  B_i'=\left \{
  \raisebox{0mm}{$b\in B_i$}
  \left|
  \parbox{2in}{\centering\text{ if } $\alpha(j)=i, \text{ then } z(b)=\sigma_j 0*$\\ \text{ if } $\beta(j)=i, \text{ then } z(b)=\sigma_j1*$ }
  \right.\right\},
\]
and where $j\in\{1,\ldots,d\}$ is such that either $\alpha(j)=i$ or $\beta(j)=i$ (exactly one such $j$ exists and exactly one of the conditions is met as $im{(\alpha)}\cup im{(\beta)}=\{1,\hdots,2d\}$).

We are given that $z|_{B_i}$ is injective, so $|z(B_i)|=|B_i| \sim L\sim 2^p$. The elements of $z(B_i)$ are binary sequences of length $p$. 
In the new subset $z(B_i')$, we are constraining the first $|\sigma_j|+1=q_j+1$ entries of the sequences, so there are $p-(q_j+1)$ free digits in each element of $z(B_i')$. 
With this, we can approximate the size of $B_i'$, $$|B_i'|=|z(B_i')|\leq 2^{p-(q_j+1)}=\frac{2^p}{2^{q_j+1}}\sim \frac{2^p}{2^{q_j}}\sim \frac{L}{2^{q_j}}\leq \frac{L}{2^{min(q_j) }}.$$

Using Lemma~\ref{lem:JustTRelations} we can compute $|A_{\bar \sigma}|$ in time $\sim \max |B_i'|\sim \frac{L}{2^{min(q_j)}}$.
In Equation \ref{eq:Aq}, $A_{\bar q}$ is written as a union of $A_{\bar\sigma}$'s where there are $2^{\sum q_i}$  possible choices for $\bar\sigma$. So, $|A_{\bar q}|$ can be computed in time $\sim 2^{\sum q_i} \left (\frac{L}{2^{min(q_j)}}\right)=2^{\sum'q_i}L$, where $\sum'$ denotes ``sum with the smallest summand omitted''. The worst case is when in $\bar{q}=(q_j)_{j=1}^d$ all but one of the entries are $p-1$ and one (the one omitted in $\sum'$) is unconstrained, and in that case the complexity is $(2^{p-1})^{d-1}L\sim L^d$.

Lastly, from Equation \ref{eq:A}, $A$ is the union of $A_{\bar q}$'s where the number of choices for $\bar q$ is $p^d \sim (\log_2 L)^d \sim 1$.
So up to $\sim$, we can compute $|A|$ with at most the complexity of the most expensive $|A_{\bar q}|$, which is $\sim L^d$.
\end{proof}

\vspace{1cm}
\noindent On behalf of all authors, the corresponding author states that there is no conflict of interest.
\bibliographystyle{alpha}
\bibliography{YarnBallBib}

\end{document}